\documentclass[12pt]{article}

\usepackage{setspace}
\usepackage{amsmath, amssymb,amsthm, amsfonts,listings}
\newtheorem{thm}{Theorem}

\newtheorem{prop}{Proposition}

\newtheorem{propo-thm}{Proposed Theorem}
\newtheorem{propo-prop}{Proposed Proposition}
\newtheorem{propo-cor}{Proposed Corollary}

\newcommand{\Ric}{\mbox{Ric}}
\newcommand{\R}{\mathbb R}
\theoremstyle{remark}
\newtheorem{remark}[thm]{Remark}

\numberwithin{equation}{section}

\newcommand{\be}{\begin{equation}}
\newcommand{\ee}{\end{equation}}
\def\p{\partial}

\def\la{\langle}
\def\ra{\rangle}
\def\lf{\left}
\def\ri{\right}
\def\Pi{\displaystyle{\mathbb{II}}}

\def\mS{\mathbb{S}}
\def\V{\mathrm{Vol}}

\begin{document}

\title{A New Monotone Quantity along the Inverse Mean Curvature Flow in $\R^n$}

\author{
Kwok-Kun Kwong\thanks{%
School of Mathematical Sciences, Monash University, Victoria 3800, Australia.
E-mail: kwok-kun.kwong@monash.edu}
\and
Pengzi Miao\thanks{%
Department of Mathematics, University of Miami, Coral Gables, FL 33146, USA.
E-mail: pengzim@math.miami.edu}
}

\date{}

\maketitle

\begin{abstract}
We find  a new monotone increasing quantity along smooth solutions to the inverse mean curvature flow in $ \R^n$.
As an application, we derive a sharp geometric inequality for mean convex, star-shaped hypersurfaces which relates
the volume enclosed by a hypersurface to a weighted total mean curvature of the  hypersurface.
\end{abstract}

%\noindent
%{Keywords:}  {inverse mean curvature flow} \\
%{MSC 2010:}  {primary 53C44; secondary 53A07}

\section{Statement of the Result}

Monotone quantities along hypersurfaces evolving under the inverse mean flow have many
applications in  geometry and relativity.
In \cite{HI}, Huisken and Ilmanen applied the monotone increasing property  of Hawking mass to give a proof of
the Riemannian Penrose Inequality. In a recent paper \cite{BHW}, Brendle, Hung and Wang discovered a monotone decreasing
quantity along the inverse mean curvature flow  in  Anti-Desitter-Schwarzschild manifolds and used it to
establish a Minkowski-type inequality for star-shaped hypersurfaces.

In this note, we provide  a new monotone increasing quantity along smooth solutions to the inverse mean curvature flow in
$ \R^n$:

\begin{thm} \label{thm: monotone}
Let $ \Sigma $ be a smooth, closed, embedded hypersurface with positive mean curvature in $ \R^n$.
Let $ I $ be an open interval and
$ X: \Sigma \times I \rightarrow \R^n $ be a smooth map satisfying
\be \label{eq-IMF}
 \frac{\p X}{\p t} = \frac{1}{H} \nu,
 \ee
where $H$ is the mean curvature of the surface $ \Sigma_t = X (\Sigma, t) $ and $ \nu $ is the outward unit normal vector to $ \Sigma_t$.
Let $ \Omega_t $ be the bounded region enclosed by $ \Sigma_t $ and
 $ r = r (x)$ be the distance from $ x $ to a fixed point $ O$. Then the function
\begin{equation}\label{eq-claim-Rn}
Q (t) = e^{ - \frac{n-2}{n-1} t} \lf[ n \mathrm{Vol}(\Omega_t) - \frac{1}{n-1} \int_{\Sigma_t} r^2 H d \mu  \ri]
\end{equation}
is monotone increasing and $Q(t) $ is  a constant function if and only if  $\Sigma_t $ is a round sphere for each $t$.
Here $ \V(\Omega)$ denotes  the volume of a bounded region $ \Omega $ and $ d \mu $ denotes the volume form on a hypersurface.
\end{thm}

As an application, we derive a sharp inequality for star-shaped hypersurfaces in $ \R^n$ which relates
the volume enclosed by a hypersurface to an $r^2$-weighted total mean curvature of the hypersurface.

\begin{thm}\label{thm: 2}
  Let $\Sigma  $ be a smooth,  star-shaped, closed hypersurface embedded in $\mathbb R^n$ with positive mean curvature.
    Then
  \begin{equation}\label{ineq: main}
   n \mathrm{Vol}(\Omega)\le \frac 1{n-1}\int_\Sigma r^2 H d \mu
   \end{equation}
 where $ \V(\Omega)$ is the volume of the region $\Omega$ enclosed by $\Sigma $,
 $r $ is the distance  to a  fixed point $O $  and $H$ is the mean curvature of $ \Sigma$.
Furthermore,  equality in \eqref{ineq: main} holds if and only if $\Sigma$ is a sphere centered at $O$.
\end{thm}

We give some remarks about Theorem \ref{thm: monotone} and Theorem \ref{thm: 2}.
The discovery of the monotonicity  of $Q(t) $ in Theorem \ref{thm: monotone}
is  motivated by the recent work of Brendle, Hung and Wang in \cite[Section 5]{BHW}.
To prove Theorem \ref{thm: monotone}, we also need a result due to Ros \cite{ros1987compact}
which was proved using Reilly's formula \cite{reilly1977applications}.
Having known $ Q(t)$ is monotone increasing,  to prove Theorem \ref{thm: 2}, it may be  attempting
 to ask whether
 $ \lim_{ t \rightarrow \infty} Q(t) = 0 $. 
We do not know if this is true because both $ \V(\Omega_t) $ and $ \int_{\Sigma_t} r^2 H d \mu $
grow like $ e^{\frac{n}{n-1} t} $ when  $\{  \Sigma_t  \} $ are spheres while
there is only a factor of $ e^{- \frac{n-2}{n-1} t}$ in
\eqref{eq-claim-Rn}.
Instead,
we take an alternate approach by first proving  Theorem \ref{thm: 2}  for a convex hypersurface $ \Sigma$.
The proof in that case again makes use of Reilly's formula.
When $\Sigma$ is merely assumed to be mean convex and star-shaped, we prove Theorem \ref{thm: 2} by reducing it to
the convex case using solutions to the inverse mean curvature flow provided by  the works of
Gerhardt \cite{gerhardt1990flow} and  Urbas \cite{urbas1990expansion}.
%%%%%%%%%%%%%%%%%%%%%
%
% Remarks added after submission
%
%%%%%%%%%%%%%%%%%%%%
If  a stronger result of Huisken and Ilmanen in \cite{HI08} is applied, Theorem \ref{thm: 2} indeed can be  shown to hold
for star-shaped surfaces with nonnegative mean curvature. We will discuss this case in the end.

\section{Proof of the Theorems}

Given a compact Riemannian manifold  $(\Omega, g)$   with boundary $\Sigma $,
we recall that Reilly's formula \cite{reilly1977applications} asserts
\be \label{eq-Reilly}
\begin{split}
& \ \int_\Omega | \nabla^2 u |^2 + \la \nabla (\Delta u), \nabla u \ra + \Ric (\nabla u, \nabla u) \ d V  \\
= & \int_{\Sigma } ( \Delta u ) \frac{\p u}{\p \nu} - \Pi (\nabla^\Sigma u, \nabla^\Sigma u) - 2 (\Delta_\Sigma u) \frac{\p u}{\p \nu}
- H \lf( \frac{\p u}{\p \nu} \ri)^2 d \mu .
\end{split}
\ee
Here $u$ is a smooth function on $ \Omega$;  $ \nabla^2 $, $\Delta$ and $\nabla$ denote the Hessian,
the Laplacian and the gradient on $\Omega$; $\Delta_\Sigma$ and $\nabla^\Sigma$ denote
the Laplacian and  the gradient on $\Sigma$ ;
 $\nu$ is the unit outward normal vector to $ \Sigma$;  $\Pi $ and $H$  are the second fundamental form and the mean curvature
 of $\Sigma $ with respect to $\nu$; and $ \Ric $ is the Ricci curvature of $g$.

To prove Theorem \ref{thm: monotone}, we need a result of Ros \cite{ros1987compact},
which was proved by choosing $\Delta u = 1 $ on $ \Omega$ and $ u = 0 $ at $ \Sigma$ in the above  Reilly's formula.

\begin{thm}[Ros \cite{ros1987compact}] \label{thm: Ros}
Let $ (\Omega, g)$ be an $n$-dimensional compact Riemannian manifold with nonnegative Ricci curvature
with  boundary $\Sigma$.   Suppose $ \Sigma$ has positive mean curvature $H$,  then
\be \label{eq-V-1overH}
n \mathrm{Vol}(\Omega) \le (n-1)  \int_\Sigma  \frac{1}{H}  d \mu
\ee
and equality holds if and only if $(\Omega, g)$ is isometric to a round ball in $ \R^n$.
\end{thm}

\begin{proof}  [Proof of Theorem \ref{thm: monotone}]
We use $'$ to denote differentiation w.r.t $t$.  Some basic formulas along
the inverse mean curvature flow \eqref{eq-IMF} in $ \R^n$ are
\begin{equation}\label{eq-evol}
H' =-\Delta_{\Sigma_t} \lf( \frac 1H \ri) -\frac {| \Pi|^2}H, \ \
d\mu'=d\mu, \ \
\V(\Omega_t)'= \int_{\Sigma_t}\frac 1H d\mu.
\end{equation}
Let $u= r^2$, then $u$ satisfies
\begin{equation}\label{eq-u}
  \nabla ^2 u = 2 g \textrm{\quad and \quad}\Delta u= 2n,
\end{equation}
where $g $ is the Euclidean metric.
Now
  \begin{equation}\label{eq-mono}
  \begin{split}
    \left(\int_{\Sigma_t} uH d\mu\right )'
        =& \int_{\Sigma_t}  (u'H + uH'+ uH)d\mu .
  \end{split}
  \end{equation}
  Let $ \la \cdot, \cdot \ra$ be the Euclidean inner product.
  By  \eqref{eq-evol}, \eqref{eq-u} and the divergence theorem, we have
\begin{equation}\label{eq-u'H}
\int_{\Sigma_t}u'H d\mu =  \int_{\Sigma _t}  \langle \nabla u, \frac{1}{H} \nu\rangle H d\mu
=  \int_{\Omega_t } \Delta u dV=2n \mathrm{Vol}(\Omega_t) .
\end{equation}
By \eqref{eq-u}, we also have
$$\Delta_{\Sigma_t} u = \Delta u - H \frac{\partial u }{\partial \nu}- \nabla ^2 u (\nu,\nu)= 2(n-1) -H\frac{\partial u}{\partial \nu}, $$
which together with \eqref{eq-evol} - \eqref{eq-u} implies
\begin{equation}\label{eq-uH'}
  \begin{split}
    \int_{\Sigma_t} u H' d\mu
    &= \int_{\Sigma_t}  \lf(- \frac{\Delta_{\Sigma_t}  u }{H}- \frac{u|\Pi|^2 }H \ri)d\mu\\
    & = \int_{\Sigma_t} \lf(-\frac {2(n-1)}H + \frac{\partial u}{\partial \nu} -  \frac{ u |\Pi|^2 }H \ri)d\mu\\
    &= - \int_{\Sigma_t} \frac {2(n-1)}H d\mu+2n \mathrm{Vol}(\Omega_t)-\int_{\Sigma_t}\frac{ u |\Pi|^2 }H d\mu.
  \end{split}
\end{equation}
Substituting  \eqref{eq-u'H} and \eqref{eq-uH'} into \eqref{eq-mono} yields
\begin{equation} \label{eq-combined}
  \begin{split}
    \left(\int_{\Sigma_t} uH d\mu\right )'
    &= 4n \mathrm{Vol}(\Omega_t ) +\int_{\Sigma_t} \lf[ -\frac { 2(n-1)}H - \frac{u|\Pi|^2} H +  uH \ri] d\mu\\
    &\le 4n \mathrm{Vol}(\Omega_t ) + \int_{\Sigma_t} \lf[ -\frac { 2(n-1)}H - \frac{uH }{n-1} +uH \ri] d\mu\\
    &= 4n \mathrm{Vol}(\Omega_t ) +\int_{\Sigma_t} \lf[  -\frac { 2(n-1)}H +\frac{n-2 }{n-1} uH \ri] d\mu\\
    &\le 4n \mathrm{Vol}(\Omega_t ) -2n \mathrm{Vol}(\Omega_t)+\frac{n-2 }{n-1}\int_{\Sigma_t}  uH d\mu\\
    &= 2n \mathrm{Vol}(\Omega_t ) +\frac{n-2 }{n-1}\int_{\Sigma_t}  uH d\mu
  \end{split}
\end{equation}
where we have used $|\Pi|^2\ge \frac{1}{n-1} H^2 $ in line 2 and Theorem \ref{thm: Ros} in line 4.
On the other hand, by  Theorem \ref{thm: Ros} again,  we have
\be \label{eq-dv}
 \mathrm{Vol}(\Omega _t)' = \int_{\Sigma _t} \frac 1H d\mu\ge \frac n{n-1}\mathrm{Vol}(\Omega_t).
 \ee
 It follows from \eqref{eq-combined} and \eqref{eq-dv} that
\begin{equation*}
  \begin{split}
    \left[ n(n-1)\mathrm{Vol}(\Omega _t) - \int_{\Sigma _t}uH d\mu\right]'
    \ge \frac{n-2}{n-1}\left[ n(n-1) \mathrm{Vol(\Omega_t)}- \int_{\Sigma _t}uH d\mu\right]
  \end{split}
\end{equation*}
or equivalently
\be
\left[ e^{-\frac{n-2}{n-1}t} \left(n \mathrm{Vol}(\Omega_t ) - \frac{1}{n-1}  \int_{\Sigma_t} r^2 H d\mu\right)\right]'\ge 0.
\ee
We conclude   that $ Q(t)$ is monotone increasing, moreover $ Q(t)$ is a constant function  if and only if equalities in
\eqref{eq-combined} and \eqref{eq-dv} hold. By Theorem \ref{thm: Ros}, we know these equalities hold  if and only if
$ \Sigma_t $ is a round sphere for all $t$.  This completes the proof of Theorem \ref{thm: monotone}.
\end{proof}

Next, we prove Theorem \ref{thm: 2} in the case that  $\Sigma $ is a convex hypersurface.

\begin{prop} \label{prop: convex}
 Let $\Sigma  $ be a smooth, closed, convex  hypersurface embedded in $\R^n$.
  Then
 \begin{equation}\label{ineq: main-convex}
  n \mathrm{Vol}(\Omega)\le \frac 1{n-1}\int_\Sigma r^2 H d \mu
  \end{equation}
where $ \V(\Omega)$ is the volume of the region $\Omega$ enclosed by $\Sigma $,
 $r $ is the distance  to a  fixed point $O $  and $H$ is the mean curvature of $ \Sigma$.
Moreover,  equality in \eqref{ineq: main-convex} holds if and only if $\Sigma$ is a sphere centered at $O$.
\end{prop}

\begin{remark}
Proposition  \ref{prop: convex} generalizes an  inequality of the first author  in \cite[Theorem 3.2 (1)]{Kwong12}.
\end{remark}

\begin{proof}
Apply Reilly's formula \eqref{eq-Reilly} to the Euclidean region $\Omega$ and choose
$ \displaystyle u = r^2 $, we have
$$
 4 n ( n -1)  \V(\Omega) =  \int_{\Sigma }   \Pi (\nabla^\Sigma u, \nabla^\Sigma u) + 2 (\Delta_\Sigma u) \frac{\p u}{\p \nu}
+ H \lf( \frac{\p u}{\p \nu} \ri)^2 d \mu
$$
where
$$
\Delta_\Sigma u = \Delta u - H \frac{\p u}{\p \nu} - \nabla^2 u (\nu, \nu)
= 2 (n -1)  - H \frac{\p u}{\p \nu}.
$$
Therefore,
\be \label{eq-Reilly-Euclidean-1}
 \int_\Sigma H \lf( \frac{\p u}{\p \nu} \ri)^2 d \mu   =  \int_{\Sigma }   \Pi (\nabla^\Sigma u, \nabla^\Sigma u) d \mu + 4n ( n-1)  \V(\Omega)  .
\ee
Since $ \Sigma$ is convex, $ \Pi (\cdot, \cdot)$ is positive definite. Hence, \eqref{eq-Reilly-Euclidean-1} implies
\be
  n (n-1) \V(\Omega) \le  \frac14 \int_\Sigma H \la \nabla (r^2), \nu \ra^2  d \mu \le  \int_\Sigma H r^2 d \mu .
\ee
When $ \displaystyle  n (n-1) \V(\Omega)  =    \int_\Sigma H r^2 d \mu$, we must have  $ \Pi(\nabla^\Sigma u, \nabla^\Sigma u) = 0 $, hence
$ \nabla^\Sigma u = 0 $. This implies  that $ u = r^2 $ is a constant on $ \Sigma$, which shows that $ \Sigma$ is a sphere
centered at $ O$.
\end{proof}

To deform a star-shaped hypersurface  to a convex hypersurface through the inverse mean curvature flow, we make use of
a special case of a general result of Gerhardt \cite{gerhardt1990flow} and   Urbas \cite{urbas1990expansion}.

 \begin{thm} [Gerhardt \cite{gerhardt1990flow} and Urbas \cite{urbas1990expansion}] \label{thm: Gerhardt and Urbas}
 Let $\Sigma $ be a smooth, closed hypersurface in $ \R^{n}$ with positive mean curvature, given by
 a smooth embedding $X_0 : \mS^{n-1} \rightarrow \R^{n} $.
 Suppose $ \Sigma$ is star-shaped with respect to a  point $ P $.
  Then the initial value problem
 \begin{equation} \label{eq-initial}
 \left\{
 \begin{split}
 \frac{ \p X}{ \p t}  = & \ \frac{1}{ H } \nu  \\
 X  (\cdot, 0) = & X_0 ( \cdot)
 \end{split}
 \right.
 \end{equation}
 has a unique smooth solution $ X : \mS^{n-1} \times [0, \infty) \rightarrow \R^{n} $, where
 $ \nu $ is the unit outer normal vector  to $\Sigma_t = X (\mS^{n-1}, t) $  and
 $H$ is the mean curvature of $\Sigma_t$.
 Moreover,  $\Sigma_t$ is  star-shaped with respect to  $P$  and
 the rescaled hypersurface  $ \widetilde{\Sigma_t }$, parametrized by $ \widetilde{X}(\cdot , t) =
  e^{- \frac{t}{n-1} } X( \cdot, t)   $, converges to a sphere  centered at $P$ in the $ \mathcal{C}^\infty$ topology
 as $ t \rightarrow \infty$.
 \end{thm}

Now we can complete the proof of Theorem \ref{thm: 2}.

\begin{proof}[Proof of Theorem \ref{thm: 2}]
By Theorem  \ref{thm: Gerhardt and Urbas},  there exists  a smooth solution $\{ \Sigma_t \}$
 to the inverse mean curvature flow with initial condition $ \Sigma$. Moreover,
the rescaled hypersurface $ \widetilde \Sigma_t = \{ e^{-\frac t{n-1}}x \ | \ x \in \Sigma_t \} $   converges exponentially fast
in the $C^\infty$ topology  to a sphere. In particular, $\widetilde \Sigma_t$ and hence $\Sigma_t$, must be convex for large  $t$.

Let $ T $ be a time when $ \Sigma_{T}$  becomes  convex.
By Proposition \ref{prop: convex}, we have
   \begin{equation*}
     n \mathrm{Vol}(\Omega _T) \le \frac{1}{n-1} \int_{\Sigma_T} r^2 H  d \mu,
   \end{equation*}
i.e.  $ Q(T) \le 0 $.
By Theorem  \ref{thm: monotone}, we know   $Q(t)$ is monotone increasing, hence
$$
Q(0) \le Q(T) \le 0
$$
which proves  \eqref{ineq: main}.

If  the equality in \eqref{ineq: main}  holds, then
$ Q(0) = 0 $. It follows from the monotonicity of $ Q(t)$ and the fact  $ Q(t) \le 0 $ for large $t$ that
$$Q(t) = 0, \ \forall \ t .$$
By Theorem \ref{thm: monotone}, this implies  that $ \Sigma_t $ is a sphere for each $t$.
By Proposition \ref{prop: convex}, $\Sigma_t$ is  a sphere centered at $O$ for large  $t$. Therefore,
we conclude that  the initial hypersurface $ \Sigma$ is a sphere centered at $O$.

\end{proof}

%%%%%%%%%%%%%%%%%%%%%%
%
%  remarks added on 11/15/12
%
%%%%%%%%%%%%%%%%%%%%%%

%\vspace{.2cm}

%\noindent {\em Remark}: By the result in [Huisken-Ilmanen, 2008 JDG]. Theorem 2 can be strengthen to smooth star-shaped
%surface with $H \ge 0$, including the rigidity statement.

\section{The case of nonnegative mean curvature}

Suppose $ \Sigma  $ is a  star-shaped hypersurface with nonnegative mean curvature in $ \R^n$.
By approximating $ \Sigma$ with  star-shaped hypersurfaces with positive mean curvature, it is not hard to see that
the inequality  \eqref{ineq: main} still holds for $ \Sigma$.
(For instance, such an approximation can be provided by the short time solution to  the mean curvature flow with initial condition $ \Sigma $.)

To see that the  rigidity part of \eqref{ineq: main} also holds for such a $ \Sigma$, we resort to a result of Huisken and Ilmanen in \cite[Theorem 2.5]{HI08}:

\begin{thm} [Huisken and Ilmanen \cite{HI08}] \label{thm: HI08}
 Let $X_0 : \mS^{n-1} \rightarrow \R^{n} $ be an  embedding such that $ \Sigma = X_0 ( \mS^{n-1} ) $ is a $C^1$, star-shaped
 hypersurface with measurable, bounded, nonnegative weak mean curvature.
  Then
 \begin{equation} \label{eq-HI08}
 \frac{ \p X}{ \p t}  =  \ \frac{1}{ H } \nu
 \end{equation}
 has a  smooth solution $ X : \mS^{n-1} \times (0, \infty) \rightarrow \R^{n} $ such that as  $ t \rightarrow 0 + $, the hypersurface
 $ \Sigma_t = X (\mS^{n-1}, t ) $ converges to $ \Sigma$ uniformly in $ C^0$.
\end{thm}

\begin{remark}
In the above theorem, if the initial surface $ \Sigma$ is assumed to be smooth, the same proof in \cite{HI08}
%(using the fact  $ \Sigma_t $ locally can be written as a graph of bounded gradient over $ \Sigma$ and the lower estimate on $H$ in \cite[Theorem 1.1]{HI08})
together with the upper estimate of $H$ for smooth solutions (c.f. \cite[(1.4)]{HI})
shows  that as $ t \rightarrow 0 +$, $ \Sigma_t $ converges to $ \Sigma$ in $W^{2,p}$ norm for any $ 1 < p <\infty$. 
On the other hand, by Theorem \ref{thm: Gerhardt and Urbas}, $\Sigma_t$ converges to a sphere in the $C^\infty$ topology after rescaling, as $t\to \infty$. In particular, $\Sigma_t$ is convex for large enough $t>0$.
\end{remark}

It follows from  Theorem \ref{thm: monotone}, Proposition \ref{prop: convex}, Theorem \ref{thm: HI08} and Remark 7 that

\begin{thm}\label{thm: nonnegative-H}
  Let $\Sigma  $ be a smooth,  star-shaped, closed hypersurface embedded in $\mathbb R^n$ with nonnegative mean curvature.
    Then
  \begin{equation}\label{ineq: nonnegative-H}
   n \mathrm{Vol}(\Omega)\le \frac 1{n-1}\int_\Sigma r^2 H d \mu
   \end{equation}
 where $ \V(\Omega)$ is the volume of the region $\Omega$ enclosed by $\Sigma $,
 $r $ is the distance  to a  fixed point $O $  and $H$ is the mean curvature of $ \Sigma$.
Furthermore,  equality in \eqref{ineq: main} holds if and only if $\Sigma$ is a sphere centered at $O$.
\end{thm}

%%%%%%%%%%%%%%%%%%%%
%
% Reference updated after submission
%
%
%%%%%%%%%%%%%%%%%%%%%

\end{document}